\documentclass[11pt, reqno]{amsart}
\usepackage{amssymb, amstext, amscd, amsmath, amsthm}
\usepackage{color}
\usepackage{hyperref}

\usepackage{xy}
\xyoption{all}

\theoremstyle{plain}
\newtheorem{theorem}{Theorem}[section]
\newtheorem{thm}[theorem]{Theorem}
\newtheorem{cor}[theorem]{Corollary}
\newtheorem{prop}[theorem]{Proposition}
\newtheorem{lem}[theorem]{Lemma}
\theoremstyle{definition}
\newtheorem{rem}[theorem]{Remark}
\newtheorem{example}[theorem]{Example}

\newtheorem{defn}[theorem]{Definition}

\newtheorem{eg}[theorem]{Example}

\newtheorem{notation}[theorem]{Notation}

\newcommand{\bN}{{\mathbb{N}}}

\newcommand{\upchi}{{\raise.35ex\hbox{\ensuremath{\chi}}}}

\newcommand{\Iso}{\operatorname{Iso}}

\newcommand{\mt}{\varnothing}

\usepackage[normalem]{ulem}
\usepackage{color}

\begin{document}
\title[Cartan Subalgebras]{Cartan Subalgebras of Topological Graph Algebras and $k$-Graph $C^*$-algebras}
\author[J. Brown]{Jonathan Brown}
\address{Jonathan Brown, Department of Mathematics, University of Dayton, 300 College Park Dayton, OH 45469-2316, USA}
\email{jonathan.henry.brown@gmail.com}
\author[H. Li]{Hui Li}
\address{Hui Li, Research Center for Operator Algebras and Shanghai Key Laboratory of Pure Mathematics and Mathematical Practice, Department of Mathematics, East China Normal University, 3663 Zhongshan North Road, Putuo District, Shanghai 200062, China}
\address{Department of Mathematics $\&$ Statistics, University of Windsor, Windsor, ON N9B 3P4, CANADA}
\email{lihui8605@hotmail.com}

\author[D. Yang]{Dilian Yang}
\address{Dilian Yang,
Department of Mathematics $\&$ Statistics, University of Windsor, Windsor, ON
N9B 3P4, CANADA}
\email{dyang@uwindsor.ca}

\thanks{The second author is the corresponding author.}
\thanks{The second author was partially supported by Research Center for Operator Algebras of East China Normal University and supported by Science and Technology Commission of Shanghai Municipality (STCSM), grant No. 13dz2260400.
}
\thanks{
The third author was partially supported by an NSERC Discovery grant.}

\begin{abstract}
In this paper, two sufficient and necessary conditions are given. The first one characterizes when the boundary path groupoid of a topological graph without singular vertices has closed interior of its isotropy group bundle, and the second one characterizes when the path groupoid of a row-finite $k$-graph without sources has closed interior of its isotropy group bundle. It follows that the associated topological graph algebra and the associated $k$-graph C*-algebra have Cartan subalgebras due to a result of Brown-Nagy-Reznikoff-Sims-Williams.

\end{abstract}


\maketitle

\section{Introduction}

In an effort to study general Cuntz-Krieger type uniqueness theorems, Brown, Nagy, Reznikoff, Sims, and Williams \cite{BNRSW}  found that for a locally compact Hausdorff \'etale groupoid $\Gamma$, the inclusion of the interior of isotropy $\mathrm{Iso}(\Gamma)^\mathrm{o}$ into $\Gamma$ induces an inclusion of the $C^*$-algebras  $C_r^*(\mathrm{Iso}(\Gamma)^\mathrm{o})\hookrightarrow C_r^*(\Gamma)$.  Furthermore they characterized when this inclusion is Cartan.

\begin{thm}[{\cite[Corollary~4.5]{BNRSW}}]\label{Cartan subalg iff closed}
Let $\Gamma$ be a locally compact \'{e}tale groupoid. Suppose that $\mathrm{Iso}(\Gamma)^\mathrm{o}$ is abelian. Then $\iota_r(C_r^*(\mathrm{Iso}(\Gamma)^\mathrm{o}))$ is a Cartan subalgebra of $C_r^*(\Gamma)$ if and only if $\mathrm{Iso}(\Gamma)^\mathrm{o}$ is closed. \end{thm}  This result has many applications. For example, it induces the Cuntz-Krieger uniqueness theorem \cite{KumjianPaskEtAl:PJM98} and the general Cuntz-Krieger uniqueness theorem \cite{Szyma'nski:IJM02}.

The motivations for Brown, et al. \cite{BNRSW}  originally came from graph algebras and their generalizations.  Indeed, many algebras coming from graphs have groupoid models, including:
\begin{itemize}
\item directed graphs \cite{KumjianPaskEtAl:JFA97},
\item topological graphs \cite{KumjianLiTwisted2}, and
\item $k$-graphs \cite{KP00}.
\end{itemize}
The difficulty of applying the full power of Theorem \ref{Cartan subalg iff closed} to these situations, is that it is unclear when the interior of the isotropy is closed in the associated groupoid.  Brown, et al. show in \cite{BNRSW} using indirect means that the interior of the isotropy is always closed in the groupoid associated to directed graphs, but they also provide an example \cite[Example~4.7]{BNRSW} in which the interior of the isotropy group bundle of the path groupoid of the $k$-graph is not closed.  Yang provided some partial results for $k$-graph $C^*$-algebras
in \cite{Yan15}, and proved the sufficiency of Theorem \ref{Cartan subalg iff closed} for $k$-graph $C^*$-algebras in \cite{YangPAMS}. 
However, there are currently no conditions intrinsic to a $k$-graph that show the associate groupoid satisfies the conditions of Theorem \ref{Cartan subalg iff closed}.

In this paper, we investigate  examples of locally compact \'{e}tale groupoids arising from directed graphs, topological graphs and higher-rank graphs. We determine when they have closed interior of the isotropy group bundle. The paper is organized as follows. In Section~2, we briefly review the background of groupoid $C^*$-algebras, graph algebras, topological graph algebras, and $k$-graph $C^*$-algebras. Recently, Kumjian and Li in \cite{KumjianLiTwisted2} proved that for a topological graph algebra without singular vertices, its associated $C^*$-algebra is isomorphic to the reduced groupoid $C^*$-algebra of the boundary path groupoid of the underlying graph. In Section~\ref{TG} we study the groupoid of a topological graph.  We produce an example (Example~\ref{ex1}) that shows the boundary path groupoid of a topological graph does not have closed interior of its isotropy group bundle in general. We then provide a sufficient and necessary condition, under which the boundary path groupoid of a topological graph without singular vertices has closed interior of its isotropy group bundle. In Section~\ref{kG} we find a sufficient and necessary condition, under which the path groupoid of a row-finite $k$-graph without sources has closed interior of its isotropy group bundle.  Finally, in Appendix~\ref{DG}, we present a short and direct proof that the graph groupoid of a row-finite directed graph without sources always has closed interior of its isotropy group bundle. This theorem can be inferred from either Section~\ref{TG} or \ref{kG}, but we include it here because of its relative simplicity.

\section{Preliminaries}

Throughout this paper, all topological spaces are assumed to be second countable; all locally compact groupoids are assumed to be second-countable locally compact Hausdorff groupoids. 
By $\mathbb{N}$ (resp. $\bN_+$), we denote the set of all nonnegative (resp. positive) integers. 
Denote by $H$ the infinite-dimensional separable Hilbert space.

\subsection{Groupoids}

A groupoid is a small category where every morphism has an inverse.  In this paper we deal exclusively with topological groupoids, that is a groupoid with a topology in which composition and inversion are continuous. For a groupoid $\Gamma$ let $\Gamma^0$ be the set of objects in $\Gamma$ which can be identified with the set of identity morphisms in $\Gamma$.  If $\Gamma$ is a groupoid there exist maps $r,s:\Gamma\to \Gamma^0$, where $r(\gamma)$ is the range of the morphism $\gamma$ and $s(\gamma)$ is the source of the morphism $\gamma$.  A locally compact groupoid is said to be \emph{\'{e}tale} if its range and source maps are both local homeomorphisms.

Let $\Gamma$ be a locally compact \'{e}tale groupoid. For $u \in \Gamma^0$, denote by $\Gamma_u:=s^{-1}(u)$; $\Gamma^u:=r^{-1}(u)$; and by $\Gamma_u^u:=\Gamma_u \cap \Gamma^u$ the \emph{isotropy group}. Define the \emph{isotropy group bundle} by $\mathrm{Iso}(\Gamma):=\bigcup_{u \in \Gamma^0}\Gamma_u^u$, which closed in $\Gamma$ and is a locally compact \'{e}tale subgroupoid of $\Gamma$. Moreover, $\Gamma$ is said to be \emph{essentially free} if the set of units whose isotropy groups are trivial is dense in $\Gamma^0$. Furthermore, the interior of $\mathrm{Iso}(\Gamma)$, denoted by $\mathrm{Iso}(\Gamma)^{\mathrm{o}}$, is open and a locally compact \'{e}tale subgroupoid of $\Gamma$.

\subsection{Topological Graphs}

Topological graphs were introduced in \cite{Katsura:TAMS04} as  a generalization of  graph algebras which adds topologies to the vertex and edge spaces. In this subsection, we review the background of topological graph algebras from \cite{Katsura:TAMS04, KumjianLiTwisted2}.

A \emph{topological graph} is a quadruple $E=(E^0,E^1,r,s)$ such that $E^0, E^1$ are locally compact Hausdorff spaces, $r:E^1 \to E^0$ is a continuous map, and $s:E^1\to E^0$ is a local homeomorphism. A  {\em directed graph} is a topological graph where $E^0$ and $E^1$ are countable and discrete. Let $E$ be a topological graph. A subset $U$ of $E^1$ is called an \emph{$s$-section} if $s\vert_{U}:U \to s(U)$ is a homeomorphism with respect to the subspace topologies. There are some specific subsets of $E^0$ analogous to the ones in directed graphs. The set of \emph{finite receivers} $E_{\mathrm{fin}}^0$ consists of all $v \in E^0$ which has an open neighborhood $N$ such that $r^{-1}(\overline{N})$ is compact. The set of \emph{sources} $E_{\mathrm{sce}}^0:=E^0 \setminus \overline{r(E^1)}$. The set of \emph{regular vertices} $E_{\mathrm{rg}}^0:=E_{\mathrm{fin}}^0 \setminus \overline{E_{\mathrm{sce}}^0}$. Moreover, the set of \emph{singular vertices} $E_{\mathrm{sg}}^0$ is defined to be the complement of $E_{\mathrm{rg}}^0$. The sets $E_{\mathrm{fin}}^0, E_{\mathrm{sce}}^0, E_{\mathrm{rg}}^0$ are all open, and the set $E_{\mathrm{sg}}^0$ is closed.

Let $E$ be a topological graph such that $E_{\mathrm{sg}}^0=\mt$. For $n \geq 2$, define
\[
E^n:=\Big\{\mu =(\mu_1,\dots,\mu_n)\in \prod_{i=1}^{n}E^1: s(\mu_i)=r(\mu_{i+1}), i=1,\dots,n-1 \Big\}
\]
endowed with the subspace topology of the product space $\prod_{i=1}^{n}E^1$. Define the \emph{finite-path space} by $E^*:=\coprod_{n=0}^{\infty} E^n$ with the disjoint union topology. Define the \emph{infinite path space} by
\[
E^\infty:=\Big\{x\in \prod_{i=1}^{\infty}E^1: s(x_i)=r(x_{i+1}), i=1,2,\dots \Big\}
\]
Denote the length of a path $\mu \in E^* \amalg E^\infty$ by $\vert\mu\vert$ and $\mu(n)=x_1\cdots x_n$ for $n\leq |\mu|$.  Endow $E^\infty$ with the subspace topology inherited from the  product space $\prod_{i=1}^{\infty}E^1$.  A basis for this topology consists of $Z(U):=\{x\in E^*: x(n)\in U\}$ where $U$ is open in $E^n$. The product topology on $E^\infty$ is locally compact Hausdorff by \cite[Definition~4.7, Lemma~4.8]{KumjianLiTwisted2} .   The one-sided shift map $\sigma:E^\infty \to E^\infty,\quad x_1x_2\cdots\mapsto x_2x_3\cdots$ is a local homeomorphism by \cite[Lemma~7.1]{KumjianLiTwisted2}.

\begin{defn}[{\cite[Definition~2.4]{Renault:00}}]\label{boundary-path gpoid}
Let $E$ be a topological graph such that $E_{\mathrm{sg}}^0=\mt$. Define the \emph{boundary path groupoid} by
\begin{align*}
\Gamma(E^\infty,\sigma):=\{(x,k-\ell,y) \in E^\infty \times \mathbb{Z} \times E^\infty: \sigma^{k}(x)=\sigma^\ell(y)\}.
\end{align*}
The range of $(x,m,y)$ is $x$ and its source is $y$, so  $\mathrm{Iso}(\Gamma(E^\infty, \sigma))=\{(x,k,x)\in \Gamma(E^\infty, \sigma)\}$.

For open subsets $U, V$ of $E^\infty$ satisfying that $\sigma^{k}$ is injective on $U$, and $\sigma^{\ell}$ is injective on $V$, denote by
\[
\mathcal{U}(U,V, k,\ell):=\{(x,k-\ell,y):x \in U, y\in V,\sigma^{k}(x)=\sigma^{\ell}(y)\}.
\]
The collection $\{\mathcal{U}(U,V,k,\ell)\}$ of subsets of $\Gamma(E^\infty,\sigma)$ as above forms an open base on $\Gamma(E^\infty,\sigma)$, and under this topology $\Gamma(E^\infty,\sigma)$ is a locally compact \'{e}tale groupoid.
\end{defn}

We give a characterization of convergent sequences in $\Gamma(E^\infty,\sigma)$. Fix a sequence $((x^t,n_t,y^t))_{i=1}^{\infty} \subset \Gamma(E^\infty,\sigma)$, and fix $(x,n,y) \in \Gamma(E^\infty,\sigma)$. Find $k,\ell \geq 0$ such that
\begin{enumerate}
\item $n=k-\ell, \sigma^{k}(x)=\sigma^{\ell}(y)$;
\item for $k',\ell' \geq 0$ satisfying that $k' \leq k, \ell' \leq \ell, n=k'-\ell',\sigma^{k'}(x)=\sigma^{\ell'}(y)$, then we have $k'=k, \ell'=\ell$.
\end{enumerate}
Then $(x^t,n_t,y^t) \to (x,n,y)$ if and only if $x^t \to x, y^t \to y$, and there exists $N \geq 1$ such that whenever $i \geq N$, we have $n_t=n$,  and $\sigma^{k}(x)=\sigma^{\ell}(y)$.

\subsection{$k$-Graphs}

$k$-graphs are generalizations of directed graphs introduced in \cite{KP00} where edges are replaced by hyper-rectangles of dimension less than or equal to $k$. In this subsection, we recall the background of $k$-graph $C^*$-algebras from \cite{KP00}.

\begin{notation}
Let $k \in \mathbb{N}_+$ and $n,m \in \mathbb{N}^k$. Denote $n \lor m:=(\max\{n_i,m_i\})_{i=1}^{k}$; and $n \land m:=(\min\{n_i,m_i\})_{i=1}^{k}$.
\end{notation}

\begin{defn}[{\cite[Definitions~1.1, 1.4]{KP00}}]
Let $k \in \mathbb{N}_+$. A small category $\Lambda$ is called a \emph{$k$-graph} if there exists a functor $d:\Lambda \to \mathbb{N}^k$ satisfying the \emph{factorization property}, that is, for $\mu\in\Lambda, n,m \in \mathbb{N}^k$ with $d(\mu)=n+m$, there exist unique $\nu,\alpha \in \Lambda$ such that $d(\nu)=n,d(\alpha)=m, s(\nu)=r(\alpha), \mu=\nu\alpha$. The functor $d$ is called the 
\emph{degree map} of $\Lambda$.  A 1-graph is a directed graph.

Let $(\Lambda_1,d_1), (\Lambda_2,d_2)$ be two $k$-graphs. A functor $f:\Lambda_1 \to \Lambda_2$ is called a \emph{morphism} if $d_2 \circ f=d_1$.

For $n \in \mathbb{N}^k$, denote by $\Lambda^n:=d^{-1}(n)$. For $A,B \subset \Lambda$, define $AB:=\{\mu\nu:\mu\in A,\nu\in B, s(\mu)=r(\nu)\}$. Then $\Lambda$ is said to be \emph{row-finite} if $\vert v\Lambda^{n}\vert<\infty$ for all $v \in \Lambda^0, n \in \mathbb{N}^k$. $\Lambda$ is said to be \emph{without sources} if $v\Lambda^{n} \neq \mt$ for all $v \in \Lambda^0, n \in \mathbb{N}^k$.
\end{defn}

\begin{example}[{\cite[Example~1.7 (ii)]{KP00}}]
Let $k \in \mathbb{N}_+$. Define $\Omega_k:=\{(p,q) \in \mathbb{N}^k \times \mathbb{N}^k: p \leq q\}$; define $\Omega_k^0:=\mathbb{N}^k$; define $r(p,q):=p$; define $s(p,q):=q$; and define $d(p,q):=q-p$. Then $(\Omega_k,\Omega_k^0,r,s)$ is a $k$-graph.
\end{example}

\begin{defn}[{\cite[Definitions~2.1, 2.4]{KP00}}]
Let $k \in \mathbb{N}_+$ and $\Lambda$ be a row-finite $k$-graph without sources. An \emph{infinite path} is a morphism from $\Omega_k$ to $\Lambda$: denote by $\Lambda^\infty$ the set of all infinite paths of $\Lambda$. For $x \in \Lambda^\infty$, for $p \in \mathbb{N}^k$, denote by $x(p):=x(0,p)$; denote by $\sigma^p(x)$ the unique element in $\Lambda^\infty$ such that $x=x(p)\sigma^p(x)$. For $\mu \in \Lambda$, denote by $Z(\mu):=\{\mu x: x \in \Lambda^\infty,s(\mu)=x(0)\}$. Endow $\Lambda^\infty$ with the topology generated by the basic open sets $\{Z(\mu):\mu \in \Lambda\}$.
\end{defn}

\begin{defn}[{\cite[Definition~2.7]{KP00}}]
Let $k \in \mathbb{N}_+$ and let $\Lambda$ be a row-finite $k$-graph without sources. Define the \emph{path groupoid} by
\begin{align*}
\mathcal{G}_\Lambda:=\{(x,p-q,y) \in \Lambda^\infty \times \mathbb{Z}^k \times \Lambda^\infty: p,q \in \mathbb{N}^k,\sigma^{p}(x)=\sigma^q(y)\}.
\end{align*}
The range of $(x,m,y)$ is $x$ and its source is $y$ so  $\mathrm{Iso}(\mathcal{G}_\Lambda)=\{(x,k,x)\in \mathcal{G}_\Lambda\}$.

For $\mu,\nu \in \Lambda$ with $s(\mu)=s(\nu)$, denote by $Z(\mu,\nu):=\{(\mu x,d(\mu)-d(\nu),\nu x):x \in \Lambda^\infty,s(\mu)=x(0)\}$. Endow $\mathcal{G}_\Lambda$ with the topology generated by the basic open sets $\{Z(\mu,\nu):\mu,\nu \in \Lambda,s(\mu)=s(\nu)\}$.
\end{defn}

By \cite[Proposition~2.8]{KP00}, $\mathcal{G}_\Lambda$ is a locally compact \'{e}tale groupoid and in particular each $Z(\mu,\nu)$ is a compact open bisection. Also $\Lambda^\infty$ is a locally compact Hausdorff space and each $Z(\mu)$ is a compact open set.

We give the characterization of convergent sequences in $\mathcal{G}_\Lambda$. Fix a sequence $((x^t,n,y^t))_{n=1}^{\infty} \subset \mathcal{G}_\Lambda$, and fix $(x,n,y) \in \mathcal{G}_\Lambda$. We have $(x^t, n,y^t) \to (x,n,y)$ if and only if for any $p,q \in \mathbb{N}^k$ satisfying that $p-q=n$ and $\sigma^p(x)=\sigma^q(y)$, there exists $N \geq 1$, such that $n \geq N \implies x^t(p)=x(p),y^t(q)=y(q), \sigma^p(x^t)=\sigma^q(y^t)$.

\section{Cartan Subalgebras of Topological Graph Algebras}
\label{TG}

In this section, we give a complete characterization of when a topological graph $E$ without singular vertices has closed interior of the isotropy group bundle of its boundary path groupoid.

The example shows that for a topological graph $E$ without singular vertices, ${\Iso(\Gamma(E^\infty, \sigma))}^{\mathrm{o}}$ is not closed in general.

\begin{eg}
\label{ex1}
Let
\[
E^0=E^1:=\{0\} \cup \left\{(\frac{1}{n},0),(-\frac{1}{n},0)\}_{n=1}^{\infty} \cup \{(-\frac{1}{n},\frac{1}{m}):m \geq n \geq 1\right\}
\]
 with the topology induced from $\mathbb{R}^2$. Define $r$ to be the identity map. Define $s$ to be the identity map on
 $\{0\} \cup \{(\frac{1}{n},0),(-\frac{1}{n},0)\}_{n=1}^{\infty}$. For $m \geq n \geq 1$, define $s(-\frac{1}{n},\frac{1}{m}):=(-\frac{1}{n},\frac{1}{m+1})$. Then $E$ is a topological graph with $E_{\mathrm{sg}}^0=\mt$.

Denote by $x:=0 0 0 \cdots$. Then $(x,1-0,x) \in \mathrm{Iso}(\Gamma(E^\infty,\sigma))$. For $n \geq 1$, denote by
\begin{align*}
x^n&:=\big(\frac{1}{n},0\big)\big(\frac{1}{n},0\big)\big(\frac{1}{n},0\big) \cdots; \\
y^n&:=\big(-\frac{1}{n},\frac{1}{n}\big) \big(-\frac{1}{n},\frac{1}{n+1}\big)\big(-\frac{1}{n},\frac{1}{n+2}\big)\cdots; \\
z^n:&=\big(-\frac{1}{n},\frac{1}{n+1}\big) \big(-\frac{1}{n},\frac{1}{n+2}\big)\big(-\frac{1}{n},\frac{1}{n+3}\big)\cdots.
\end{align*}
Notice that $(x^n,1-0,x^n) \in \mathrm{Iso}(\Gamma(E^\infty,\sigma))^{\mathrm{o}}$ and $(x^n,1-0,x^n) \to (x,1-0,x)$, which implies that $(x,1-0,x) \in \overline{\mathrm{Iso}(\Gamma(E^\infty,\sigma))^{\mathrm{o}}}$. On the other hand, $(y^n,1-0,z^n) \notin \mathrm{Iso}(\Gamma(E^\infty,\sigma))$ and $(y^n,1-0,z^n) \to (x,1-0,x)$. So $(x,1-0,x) \notin \mathrm{Iso}(\Gamma(E^\infty,\sigma))^{\mathrm{o}}$. Therefore ${\Iso(\Gamma(E^\infty, \sigma))}^{\mathrm{o}}$ is not closed.
\end{eg}

\begin{notation}\label{define C^n B^n A^n}
Let $E$ be a topological graph. For $n \geq 1$, denote by $C^n$ the set of cycles in $E^n$, which is a closed subset of $E^n$. For $k \geq 1$, $n \geq 1$, $\mu \in C^n$, and for an open neighborhood $N$ of $\mu$ denote by $k \mu:=\underbrace{\mu \cdots \mu}_{k}$; and denote by $k N:=\underbrace{N \times\dots\times N}_{k}$. Denote by $B^n$ the set of all cycles $\mu$ in $C^n$ satisfying that there exist $k \geq 1$ and an open neighborhood $N$ of $\mu$ such that
\begin{enumerate}
\item\label{no path connecting s(alpha),s(beta)} for any distinct $\alpha,\beta \in k N$, there are no paths in $N$ connecting $s(\alpha),s(\beta)$; and
\item\label{no cycles with entrance of base pt s(alpha)} for any $\alpha \in k N$, there are no cycles in $N$ with entrances and of base point $s(\alpha)$.
\end{enumerate}
\end{notation}

\begin{rem}
By Condition (2) of Notation~\ref{define C^n B^n A^n}, for each $n \geq 1$, each cycle in $B^n$ has no entrances.
\end{rem}

We show in Theorem~\ref{T:iff} below that $\text{Iso}(\Gamma(E^\infty, \sigma))^\circ$ is closed in $\Gamma(E^\infty, \sigma)$ if and only if $B^n$ is closed for all $n$. 

\begin{rem}
\label{R:dg}
If $E$ is a directed graph, i.e. a discrete topological graph then for a cycle $\mu$ in $E$ we can take $N=\{\mu\}$; in this case Condition~\eqref{no path connecting s(alpha),s(beta)} is vacuous and Condition~\eqref{no cycles with entrance of base pt s(alpha)} says that $\mu$ does not have an entrance.  So for directed graphs, $B^n$ consists of cycles of length $n$ without entrances.  Furthermore, since $E$ is discrete, $E^n$ is discrete and$B^n$ is closed. Therefore, by Theorem \ref{T:iff} below, $\text{Iso}(\Gamma(E^\infty, \sigma))^\circ$ is always closed.
\end{rem}

\begin{lem}\label{criterion of (mu,z,mu) in ISO^o}
Let $E$ be a topological graph such that $E_{\mathrm{sg}}^0=\mt$, and let $(x,n,x) \in \mathrm{Iso}(\Gamma(E^\infty,\sigma))$. Pick $p,q \geq 0$ with $p-q=n, \sigma^p(x)=\sigma^{q}(x)$.
\begin{enumerate}
\item If $n>0$, then $(x,n,x) \in \mathrm{Iso}(\Gamma(E^\infty,\sigma))^{\mathrm{o}}$ if and only if $x_{q+1}\cdots x_{p} \in B^n$.
\item   If $n<0$, then $(x,n,x) \in \mathrm{Iso}(\Gamma(E^\infty,\sigma))^{\mathrm{o}}$ if and only if $x_{p+1}\cdots x_{q} \in B^n$.
\end{enumerate}
\end{lem}

\begin{proof}
We may assume that $n>0$ and for the case $n<0$ the argument is similar. Suppose that $(x,n,x) \in \mathrm{Iso}(\Gamma(E^\infty,\sigma))^{\mathrm{o}}$. Then there exist open subsets $U,V$ of $E^*$ such that $(x,n,x) \in \mathcal{U}(Z(U),Z(V),p,q) \subset \mathrm{Iso}(\Gamma(E^\infty,\sigma))$.  Since $U\subset E^{n_U}, V\subset E^{n_V}$ for some $n_U, n_V$, picking $M>\max\{p,n_U, n_V\}$ we see that for $m \geq M$ there exists an open subset $W \subset E^m$ such that $(x,n,x) \in \mathcal{U}(Z(W),Z(W),p,q) \subset \mathcal{U}(Z(U),Z(V),p,q)$. Fix $m \geq M$. Write $m=q+kn+l$, where $k \geq 1, 0 \leq l<n$. We may choose $W$  such that
\begin{align*}
W=(W_0 \times k(W_1 \times\dots\times W_n) \times (W_1 \times\dots\times W_l)) \cap E^m,
\end{align*}
where $W_0$ is an open neighborhood of $\mu_1 \cdots \mu_q; W_1 \times\cdots\times W_n$ is an open neighborhood of the cycle $\mu_{q+1} \cdots \mu_p; W_0,W_1,\ldots, W_n$ are open $s$-sections; and $r(W_1) \subset s(W_0)$. Let
\begin{align*}
W'=(W_0 \times (k+1)(W_1 \times\dots\times W_n) \times (W_1 \times\dots\times W_l))  \cap E^{m+n}.
\end{align*}
Then $\mathcal{U}(Z(W'),Z(W),p,q) \subset \mathcal{U}(Z(W),Z(W),p,q)$.

We claim that $\mathcal{U}(Z(W'),Z(W),p,q)=\mathcal{U}(Z(W'),Z(W),m+n,m)$.  First we show $\mathcal{U}(Z(W'),Z(W),m+n,m)\subset \mathcal{U}(Z(W'),Z(W),p,q)$. To see this, fix $(y,n,z) \in \mathcal{U}(Z(W'),Z(W),m+n,m)$. Then
\begin{equation}
\label{eq1}
\sigma^{m-q}(\sigma^{q+n}(y))=\sigma^{m+n}(y)=\sigma^{m}(z)=\sigma^{m-q}(\sigma^q(z)).
\end{equation}
Since $\sigma^{q+n}(y),\sigma^q(z) \in Z(k(W_1 \times\dots\times W_n) \times (W_1 \times\dots\times W_l))$, by \cite[Lemma~7.1]{KumjianLiTwisted2} $\sigma^{q+n}(y)=\sigma^q(z)$. So $(y,n,z) \in \mathcal{U}(Z(W'),Z(W),p,q)$. Hence $\mathcal{U}(Z(W'),Z(W),m+n,m)\subset \mathcal{U}(Z(W')$, $Z(W),p,q)$.   Permuting equation~\eqref{eq1} we see that $\mathcal{U}(Z(W'),Z(W),p,q) \subset \mathcal{U}(Z(W'),Z(W),m+n,m)$ as well. 

We choose an arbitrary $m=q+k_0 n \geq M$ for some $k_0 \geq 1$. We claim $k_0$ and $N:=W_1 \times\dots\times W_n$ satisfy Conditions~(\ref{no path connecting s(alpha),s(beta)}), (\ref{no cycles with entrance of base pt s(alpha)}) of Notation~\ref{define C^n B^n A^n}.

For condition~(\ref{no path connecting s(alpha),s(beta)}), if $\alpha,\beta\in k_0 N$ and $\gamma\in N$ with $s(\alpha\gamma)=s(\beta)$ pick $y\in E^\infty$ with $r(y)=s(\beta)$ and $\nu_\alpha, \nu_\beta\in W_0$ with $s(\nu_\omega)=r(\omega)$. Then
\[(\nu_\alpha\alpha\gamma y,n, \nu_\beta\beta y)\in \mathcal{U}(Z(W'),Z(W),m+n,m)\subset \mathrm{Iso}(\Gamma(E^\infty,\sigma)).
\]  Therefore $\nu_\alpha\alpha\gamma y=\nu_\beta\beta y$ and so $\nu_\alpha=\nu_\beta$ and $\alpha=\beta$.

For condition  (\ref{no cycles with entrance of base pt s(alpha)}), suppose $\alpha\in k_0 N$ and $\gamma\in s(\alpha)N$ is a cycle with an entrance $\delta$.  Write $\gamma=\gamma' \gamma''$ where $r(\delta)=r(\gamma'')$ and $\gamma''(|\delta|)\neq \delta$.  Pick $y\in s(\delta)E^\infty$ and $\nu\in W_0$.  Consider $z=\gamma'\delta y$.  Then
\[
(\nu \alpha \gamma z,n,\nu \alpha  z)\in \mathcal{U}(Z(W'),Z(W),m+n,m)\subset \mathrm{Iso}(\Gamma(E^\infty,\sigma)).
\]  Therefore $\gamma\gamma'\delta=\gamma'\delta$ contradicting that $\delta$ is an entrance.

Conversely, suppose that $x_{q+1}\cdots x_{p} \in B^n$. Then there exist $k \geq 1$ and an open neighborhood $N$ of $x$ satisfying Conditions~(\ref{no path connecting s(alpha),s(beta)}), (\ref{no cycles with entrance of base pt s(alpha)}) of Notation~\ref{define C^n B^n A^n}. Choose an arbitrary open $s^q$-section $O$ containing $x_1 \cdots x_q$. Let $W:=(O \times k N) \cap E^{q+kn}$ and let $W':=(O \times (k+1) N) \cap E^{q+(k+1)n}$. Now
\begin{align*}
(x,n,x) &\in \mathcal{U}(Z(W'),Z(W),p,q) \subset \mathcal{U}(Z(W'),Z(W),p+kn,q+kn)\\
&\subset \mathrm{Iso}(\Gamma(E^\infty,\sigma)).
\end{align*}
So $(x,n,x) \in \mathrm{Iso}(\Gamma(E^\infty,\sigma))^{\mathrm{o}}$.
\end{proof}

\begin{thm}
\label{T:iff}
Let $E$ be a topological graph with $E_{\mathrm{sg}}^0=\mt$. Then $\mathrm{Iso}(\Gamma(E^\infty,\sigma))^{\mathrm{o}}$ is closed in $\Gamma(E^\infty,\sigma)$ if and only if $B^n$ is closed in $E^n$ for all $n \geq 1$.
\end{thm}

\begin{proof}
Suppose that $\mathrm{Iso}(\Gamma(E^\infty,\sigma))^{\mathrm{o}}$ is closed in $\Gamma(E^\infty,\sigma)$. Fix $n \geq 1$. Fix $(\mu^k)_{k=1}^{\infty} \subset B^n$ which is convergent to $\mu \in E^n$ ($\mu$ is a cycle). Then $(\mu^k \mu^k \cdots,n-0,\mu^k \mu^k \cdots) \to (\mu\mu\cdots,n-0,\mu\mu\cdots)$.  Fix $k \geq 1$. By the definition of $B^n$ in Notation~\ref{define C^n B^n A^n} and by Lemma~\ref{criterion of (mu,z,mu) in ISO^o}, we have $(\mu^k \mu^k \cdots,n-0,\mu^k \mu^k \cdots) \in \mathrm{Iso}(\Gamma(E^\infty,\sigma))^{\mathrm{o}}$. Since $\mathrm{Iso}(\Gamma(E^\infty,\sigma))^{\mathrm{o}}$ is closed, $(\mu\mu\cdots,n-0,\mu\mu\cdots) \in \mathrm{Iso}(\Gamma(E^\infty,\sigma))^{\mathrm{o}}$. Again by the definition of $B^n$ in Notation~\ref{define C^n B^n A^n} and by Lemma~\ref{criterion of (mu,z,mu) in ISO^o}, $\mu \in B^n$. So $B^n$ is closed.

Conversely, suppose that $B^n$ is closed in $E^n$ for all $n \geq 1$. Fix a convergent net $(x^k,n_k,x^k) \subset \mathrm{Iso}(\Gamma(E^\infty,\sigma))^{\mathrm{o}}$ with the limit $(x,n,x)$. We may assume that $n_k=n$ for all $k$ and we take arbitrary $p,q \geq 0$ such that $p-q=n; \sigma^{p}(x^k)=\sigma^{q}(x^k); \sigma^{p}(x)=\sigma^{q}(x)$. We may further assume that $n>0$ since the case $n<0$ shares a symmetric proof. By Lemma~\ref{criterion of (mu,z,mu) in ISO^o}, $x_{q+1}^k \cdots x_{p}^k \in B^n$. Since $x_{q+1}^k \cdots x_{p}^k \to x_{q+1} \cdots x_p$ and $B^n$ is closed, $x_{q+1} \cdots x_p \in B^n$. Again by Lemma~\ref{criterion of (mu,z,mu) in ISO^o}, $(x,n,x) \in \mathrm{Iso}(\Gamma(E^\infty,\sigma))^{\mathrm{o}}$.
\end{proof}

Since $B^n$ is closed for a directed graph,  Theorem~\ref{T:iff} gives Proposition~\ref{P:IsoE} (see Remark \ref{R:dg}). 

\begin{notation}
Let $E$ be a topological graph. For $n \geq 1$, denote by

\[
V_n:=\left\{v \in E^0:
\begin{matrix} \text{there is an open neighborhood of } v \text{ consisting of}\\
\text{base points of cycles without entrances in } E^n
\end{matrix}
\right\}.
\]
\end{notation}

Notice that $V_n$ is an open subset of $E^0$ for all $n \geq 1$.

\begin{thm}\label{the relationship between B^n and V_n}
Let $E$ be a topological graph. Fix $n \geq 1$. Then $r^n(B^n)=V_n,B^n=(r^n)^{-1}(V_n)$. Hence $B^n$ is an open subset of $E^n$. Furthermore, suppose that $E_{\mathrm{sg}}^0=\mt$. Then $B^n$ is closed if and only if $V_n$ is closed. In these cases, $\mathrm{Iso}(\Gamma(E^\infty,\sigma))^{\mathrm{o}}$ is closed in $\Gamma(E^\infty,\sigma)$.
\end{thm}
\begin{proof}
First of all, we show that $r^n(B^n)=V_n$. Fix $\mu \in B^n$. Then there exist $k \geq 1$ and an open neighborhood $N$ of $\mu$ satisfying Conditions~(\ref{no path connecting s(alpha),s(beta)}), (\ref{no cycles with entrance of base pt s(alpha)}) of Notation~\ref{define C^n B^n A^n}. Then $W:=s^{(k+1)n}((k+1)N \cap E^{(k+1)n})$ is an open neighborhood of $r^n(\mu)$. For $v\in W$, there exists $\alpha=\alpha^{(1)}\cdots\alpha^{(k)}\alpha^{(k+1)} \in (k+1)N \cap E^{(k+1)n}$, where $\alpha^{(1)},\dots, \alpha^{(k)},\alpha^{(k+1)} \in N$, such that $s^{(k+1)n}(\alpha)=v$. So $\alpha^{(k+1)}$ connects $s^{kn}(\alpha^{(1)}\cdots\alpha^{(k)})=s^{kn}(\alpha^{(2)}\cdots\alpha^{(k+1)})$. By Condition~(\ref{no path connecting s(alpha),s(beta)}) of Notation~\ref{define C^n B^n A^n}, $\alpha^{(1)}\cdots\alpha^{(k)}=\alpha^{(2)}\cdots\alpha^{(k+1)}$. So $\alpha^{(1)}=\dots=\alpha^{(k)}=\alpha^{(k+1)}$ is a cycle in $E^n$. By Condition~(\ref{no cycles with entrance of base pt s(alpha)}) of Notation~\ref{define C^n B^n A^n}, $\alpha^{(1)}=\dots=\alpha^{(k)}=\alpha^{(k+1)}$ has no entrances. So $r^n(\mu) \in V_n$ and $r^n(B^n) \subset V_n$. Conversely, fix $v \in V_n$. Then there exists an open neighborhood $W$ of $v$ consisting of base points of cycles without entrances in $E^n$. So $(r^n)^{-1}(W)$ is an open neighborhood of $(r^n)^{-1}(v)$ consisting of cycles without entrances. It is straightforward to see that $(r^n)^{-1}(v) \in B^n$. So $r^n(B^n)=V_n$. It follows immediately that $B^n=(r^n)^{-1}(V_n)$. Hence $B^n$ is an open subset of $E^n$ because $V_n$ is open.

Finally, suppose  $E_{\mathrm{sg}}^0=\mt$ and $B^n$ is closed. Since $E_{\mathrm{sg}}^0=\mt, r^n$ is proper, and so $r^n$ is closed. Since $r^n(B^n)=V_n, V_n$ is closed. Conversely, suppose that $V_n$ is closed. Since $B^n=(r^n)^{-1}(V_n), B^n$ is closed. \end{proof}

For the next corollary we need a definition.

\begin{defn}[{\cite[Definition~5.4]{Katsura:TAMS04}}] A topological graph E is topologically free if the set of
base points of loops without entrances has empty interior.
\end{defn}

\begin{cor}\label{top free iff B^n=emptyset}
Let $E$ be a topological graph such that $E_{\mathrm{sg}}^0=\mt$. Then the following are equivalent.
\begin{enumerate}
\item\label{E top free} $E$ is topologically free;
\item\label{all B^n=emptyset} $B^n=\mt$ for all $n \geq 1$;
\item\label{all V_n=emptyset} $V_n=\mt$ for all $n \geq 1$;
\item\label{ess free groupoid} $\Gamma(E^\infty,\sigma)$ is essentially free.
\end{enumerate}
In these cases, $\mathrm{Iso}(\Gamma(E^\infty,\sigma))^{\mathrm{o}}$ is closed in $\Gamma(E^\infty,\sigma)$.
\end{cor}
\begin{proof}
(\ref{E top free})$\Leftrightarrow$(\ref{all B^n=emptyset}). Firstly, suppose that $E$ is topologically free. Suppose that there exists $n \geq 1$ such that $B^n \neq \mt$. Since $B^n=(r^n)^{-1}(V_n)$ by Theorem~\ref{the relationship between B^n and V_n}, $V_n \neq \mt$ which is a contradiction. So $B^n=\mt$ for all $n \geq 1$. Conversely, suppose that $B^n=\mt$ for all $n \geq 1$. 
To the contrary, assume that $E$ is not topologically free. By Theorem~\ref{the relationship between B^n and V_n}, $V_n=\mt$ for all $n \geq 1$. Since $E$ is not topologically free, by \cite[Proposition~6.12]{Katsura:ETDS06}, there exist a nonempty open subset $V \subset E^0$ and $n\in \mathbb{N}_+$ such that $V$ consists of base points of cycles in $E^n$ without entrances, which is a contradiction. So $E$ is topologically free.

(\ref{all B^n=emptyset})$\Leftrightarrow$(\ref{all V_n=emptyset}) follows from Theorem~\ref{the relationship between B^n and V_n}.

(\ref{E top free})$\Leftrightarrow$(\ref{ess free groupoid}) follows from \cite[Proposition~3.7]{Li15}.
\end{proof}

\begin{cor}
Let $E$ be a topological graph such that $E_{\mathrm{sg}}^0=\mt$ and $\mathcal{O}(E)$ is simple. Then $\mathrm{Iso}(\Gamma(E^\infty,\sigma))^{\mathrm{o}}$ is closed in $\Gamma(E^\infty,\sigma)$.
\end{cor}
\begin{proof}
This follows  from Corollary~\ref{top free iff B^n=emptyset} and \cite[Theorem 8.12]{Katsura:ETDS06}.
\end{proof}

\section{Cartan Subalgebras of $k$-graph Algebras}
\label{kG}

In this section, we characterize of when the interior of the isotropy of a path groupoid of a row-finite $k$-graph without sources $\Lambda$ is closed.

\begin{defn}[{\cite[Definition~4.3]{MR3150172}}]
Let $k \in \mathbb{N}_+$, let $\Lambda$ be a row-finite $k$-graph without sources. Then a pair $(\mu,\nu) \in \Lambda\times\Lambda$ is called a \emph{cycline pair} if $s(\mu)=s(\nu)$ and $\mu x=\nu x$ for all $x \in s(\mu)\Lambda^\infty$.
\end{defn}

The following lemma is stated without proof in \cite[Remark~4.11]{MR3150172}.

\begin{lem}\label{Iso^o and cycline pair}
Let $k \geq 1$ and let $\Lambda$ be a row-finite $k$-graph without sources. Then
\[
\mathrm{Iso}(\mathcal{G}_\Lambda)^{\mathrm{o}}=\big\{(x,p-q,x) \in \mathcal{G}_\Lambda:\sigma^p(x)=\sigma^q(x),(x(p),x(q)) \text{ is a cycline pair}\big\}.
\]
\end{lem}
\begin{proof}
First of all, fix $(x,n,x) \in \mathrm{Iso}(\mathcal{G}_\Lambda)^{\mathrm{o}}$. Then there exist $\mu,\nu \in \Lambda$ such that $(x,n,x) \in Z(\mu,\nu) \subset \mathrm{Iso}(\mathcal{G}_\Lambda)$. So $\sigma^{d(\mu)}(x)=\sigma^{d(\nu)}(x),(x,n,x)=(\mu \sigma^{d(\mu)}(x),d(\mu)-d(\nu),\nu \sigma^{d(\nu)}(x))$, and $(\mu,\nu)$ is a cycline pair.

Conversely, fix $(x,p-q,x) \in \mathcal{G}_\Lambda$ such that $\sigma^p(x)=\sigma^q(x)$ and $(x(p),x(q))$ is a cycline pair. Then $(x,p-q,x) \in Z(x(p),x(q)) \subset \mathrm{Iso}(\mathcal{G}_\Lambda)$. So $(x,p-q,x) \in \mathrm{Iso}(\mathcal{G}_\Lambda)^{\mathrm{o}}$.
\end{proof}

\begin{notation}\label{define Lambda_p,q p neq q}
Let $k \geq 1$ and let $\Lambda$ be a row-finite $k$-graph without sources.  For $p \neq q \in \mathbb{N}^{k}$, denote by $\Lambda^{\infty}_{p,q}$ the set consisting of $x \in \Lambda^\infty$ satisfying the following properties:
\begin{enumerate}
\item\label{sigma^p(x)=sigma^q(x)} $\sigma^{p}(x)=\sigma^{q}(x)$;
\item\label{per and ape exist at the same time} for any $p', q' \in \mathbb{N}^k$ with $p'-p=q'-q\in \bN^k$, then $(x(p'),x(q'))$ is not a cycline pair, and   $(x(p')\mu,x(q')\mu)$ is a cycline pair for some $\mu \in \Lambda$.
\end{enumerate}
\end{notation}

\begin{thm}
\label{P:k-graphiff}
Let $k \geq 1$ and $\Lambda$ be a row-finite $k$-graph without sources. Then $\mathrm{Iso}(\mathcal{G}_\Lambda)^{\mathrm{o}}$ is closed if and only if $\Lambda^{\infty}_{p,q}=\mt$ for all $p \neq q \in \mathbb{N}^{k}$.
\end{thm}
\begin{proof}
First of all, suppose that $\mathrm{Iso}(\mathcal{G}_\Lambda)^{\mathrm{o}}$ is closed. Suppose that there exist $p \neq q \in \mathbb{N}^{k}$ such that $\Lambda^{\infty}_{p,q}\neq\mt$, for a contradiction. Fix $x \in \Lambda^{\infty}_{p,q}$. For $n \geq 1$, let $p_n:=p+(n,\dots,n)$ and $q_n:=q+(n,\dots,n)$. By Condition~(\ref{per and ape exist at the same time}) of Notation~\ref{define Lambda_p,q p neq q}, for $n \geq 1$ there exist $y_n \in s(x(p_n))\Lambda^\infty$ and $\mu_n \in s(x(p_n))\Lambda$ such that $x(p_n)y_n \neq x(q_n)y_n$ and $(x(p_n)\mu_n,x(q_n)\mu_n)$ is a cycline pair. For $n \geq 1$, take an arbitrary $z_n \in s(\mu_n)\Lambda^\infty$. Then $(x(p_n)y_n,p-q,x(q_n)y_n) \to (x,p-q,x)$, and $(x(p_n)\mu_n z_n,p-q,x(q_n)\mu_n z_n) \to (x,p-q,x)$. However, $(x(p_n)y_n,p-q,x(q_n)y_n) \notin \mathrm{Iso}(\mathcal{G}_\Lambda)$; and by Lemma~\ref{Iso^o and cycline pair} one has
$(x(p_n)\mu_n z_n,p-q,x(q_n)\mu_n z_n) \in \mathrm{Iso}(\mathcal{G}_\Lambda)^{\mathrm{o}}$. This is a contradiction. So $\Lambda^{\infty}_{p,q}=\mt$ for all $p \neq q \in \mathbb{N}^{k}$.

Conversely, suppose that $\Lambda^{\infty}_{p,q}=\mt$ for all $p \neq q \in \mathbb{N}^{k}$. To the contrary, let us assume that $\mathrm{Iso}(\mathcal{G}_\Lambda)^{\mathrm{o}}$ is not closed. 
Then $\overline{\mathrm{Iso}(\mathcal{G}_\Lambda)^{\mathrm{o}}} \setminus \mathrm{Iso}(\mathcal{G}_\Lambda)^{\mathrm{o}} \neq \mt$. Fix $(x,p-q,x) \in \overline{\mathrm{Iso}(\mathcal{G}_\Lambda)^{\mathrm{o}}} \setminus \mathrm{Iso}(\mathcal{G}_\Lambda)^{\mathrm{o}}$. We may assume that $p \neq q \in \mathbb{N}^{k}$ and $\sigma^p(x)=\sigma^q(x)$. Then there exist
 a sequence $(y_n,p-q,y_n)$ in $\mathrm{Iso}(\mathcal{G}_\Lambda)^{\mathrm{o}}$ converging to  $(x,p-q,x)$, and another sequence $(z_n,p-q,w_n)$ with $z_n\ne w_n$ for all $n\ge 1$
 also converging to $(x,p-q,x)$.
By Lemma~\ref{Iso^o and cycline pair} for $n \geq 1$ there exist $p_n, q_n \in \mathbb{N}^k$ with $p_n-q_n=p-q$ such that $\sigma^{p_n}(y_n)=\sigma^{q_n}(y_n)$ and $(y_n(p_n),y_n(q_n))$ is a cycline pair. Fix $p', q' \in \mathbb{N}^k$ with $p'-p=q'-q\in\bN^k$. Then there exists $N \geq 1$ such that 
\begin{align*}
&y_N(p')=x(p'),\ y_N(q')=x(q'),\ \sigma^{p'}(y_N)=\sigma^{q'}(y_N),\\
&z_N(p')=x(p'),\ w_N(q')=x(q'),\ \sigma^{p'}(z_N)=\sigma^{q'}(w_N).
\end{align*}
Since $(y_N(p_N),y_N(q_N))$ is a cycline pair and $p_N \lor p'-q_N \lor q'=p-q$, we obtain a cycline pair
\begin{align*}
&(y_N(p_N)y_N(p_N,p_N \lor p'),y_N(q_N)y_N(q_N,q_N \lor q'))
\\&=(y_N(p')y_N(p',p_N \lor p'),y_N(q')y_N(q',q_N \lor q'))
\\&=(x(p')y_N(p',p_N \lor p'),x(q')y_N(q',q_N \lor q')).
\end{align*}
But we also have that $x(p')\sigma^{p'}(z_N)=z_N \neq w_N=x(q')\sigma^{q'}(w_N)$ which implies that $(x(p'),x(q'))$ is not a cycline pair. Hence $x \in \Lambda^{\infty}_{p,q}$ which is a contradiction. Therefore $\mathrm{Iso}(\mathcal{G}_\Lambda)^{\mathrm{o}}$ is closed.
\end{proof}

As an immediate consequence, Theorem~\ref{P:k-graphiff} generalizes Proposition~\ref{P:IsoE}.

\begin{example}
We consider \cite[Example~4.7]{BNRSW}. Denote by $x:=(e_b e_r)^\infty$. It was shown that $\gamma:=((e_b e_r)^\infty,(1,-1),(e_b e_r)^\infty) \notin \mathrm{Iso}(\mathcal{G}_\Lambda)^{\mathrm{o}}$. For each $p', q' \in \mathbb{N}^k$ with $p' \geq (1,0),q' \geq (0,1), p'-q'=(1,-1)$, let $y:=\beta_b (g_b g_r h_b h_r)^\infty$, then $y$ is not aperiodic, $r(y)=s(x(p'))$ and $x(p')y \neq x(q')y$. This example illustrates that in Condition~(\ref{per and ape exist at the same time}) of Notation~\ref{define Lambda_p,q p neq q},
for $x \in \Lambda^{\infty}_{p,q}$, the infinite path $y \in s(x(p'))\Lambda^{\infty}$ such that $x(p')y \neq x(q')y$ might be periodic.
\end{example}
\appendix
\section{Cartan Subalgebras of Graph Algebras}
\label{DG}

Combining \cite[Corollary~4.5]{BNRSW} and \cite[Theorem~3.6]{MR2927813} we get for any row-finite directed graph without sources $E, \Iso(\Gamma(E^\infty,\sigma))^{\mathrm{o}}$ is closed in the graph groupoid $\Gamma(E^\infty,\sigma)$. In this appendix, we provide a direct proof of this result by investigating the graph groupoid of directed graphs.  Given $\alpha,\beta\in E^*$ with $s(\alpha)=s(\beta)$, denote by $Z(\alpha,\beta):=\mathcal{U}(Z(Z(\alpha),Z(\beta),|\alpha|,|\beta|)$.  The $Z(\alpha,\beta)$ form a basis for the topology on $\Gamma(E^\infty,\sigma)$.

\begin{prop}\label{Z(alpha,beta) cap Iso(Gamma(E^infty,sigma))=0,1}
Let $E$ be a row-finite directed graph without sources, let $\alpha,\beta$ with $|\alpha|\neq |\beta|$. Then $Z(\alpha,\beta)\cap \Iso(\Gamma(E^\infty,\sigma))$ is either empty or a singleton.
\end{prop}
\begin{proof}
If $Z(\alpha,\beta)\cap \Iso(\Gamma(E^\infty,\sigma))$ is empty, then we are done. Suppose that $Z(\alpha,\beta)\cap \Iso(\Gamma(E^\infty,\sigma))$ is nonempty. Since $|\alpha|\neq |\beta|$, without loss of generality, we assume that $\vert\alpha\vert < \vert\beta\vert$. Fix $(\alpha x,\vert\alpha\vert - \vert\beta\vert,\beta x) \in Z(\alpha,\beta)\cap \Iso(\Gamma(E^\infty,\sigma))$. Then $\alpha x=\beta x$. Since $\vert\alpha\vert < \vert\beta\vert, \beta=\alpha\gamma$, where $\gamma$ is a cycle in $E^{\vert\beta\vert-\vert\alpha\vert}$. So $x=\gamma x$, which implies that $x=\gamma\gamma\cdots$. Therefore $Z(\alpha,\beta)\cap \Iso(\Gamma(E^\infty,\sigma))$ is a singleton.
\end{proof}

\begin{prop}
\label{P:IsoE}
Let $E$ be a row-finite directed graph without sources. Then $\Iso(\Gamma(E^\infty,\sigma))^{\mathrm{o}}$ is closed.
\end{prop}
\begin{proof}
Suppose that $\gamma_i\to \gamma$ with $\gamma_i\in\Iso(\Gamma(E^\infty,\sigma))^{\mathrm{o}}$. Since $\Iso(\Gamma(E^\infty,\sigma))$ is closed, $\gamma \in \Iso(\Gamma(E^\infty,\sigma))$. If $\gamma \in \Gamma(E^\infty,\sigma)^{(0)}$ then $\gamma \in \Iso(\Gamma(E^\infty,\sigma))^{\mathrm{o}}$ because $\Gamma(E^\infty,\sigma)$ is \'{e}tale. Now suppose  $\gamma \notin \Gamma(E^\infty,\sigma)^{(0)}$. Then there exist $\alpha,\beta$ with $\vert\alpha\vert\neq\vert\beta\vert$ such that $\gamma \in Z(\alpha,\beta)$. Since $\gamma_i \to \gamma$, there exists $i_0 \geq 1$ such that $\gamma_{i_0} \in Z(\alpha,\beta)$. By Proposition~\ref{Z(alpha,beta) cap Iso(Gamma(E^infty,sigma))=0,1}, $Z(\alpha,\beta) \cap \Iso(\Gamma(E^\infty,\sigma))$ is a singleton. So $\gamma=\gamma_{i_0} \in \Iso(\Gamma(E^\infty,\sigma))^{\mathrm{o}}$. Hence $\Iso(\Gamma(E^\infty,\sigma))^{\mathrm{o}}$ is closed.
\end{proof}

\section*{Acknowledgments}

This work was initiated during D. Yang's visiting J. Brown in October 2015. D. Yang thanks J. Brown's invitation and warm hospitality. H. Li thanks J. Brown and D. Yang for the valuable discussions and their patient correspondences.

\end{document}